\pdfoutput=1
\documentclass[a4paper, british, oneside]{amsart}

\usepackage[utf8]{inputenc}
\usepackage[T1]{fontenc}

\usepackage[english]{babel}

\usepackage{amsmath,amsthm,amssymb}
\usepackage{geometry}
\usepackage[all]{xy}
\usepackage[hidelinks]{hyperref}
\usepackage{microtype}

\renewcommand{\MR}[1]{} % mathscinet-Referenzen abstellen

\newcommand{\Z}{\mathbb{Z}}
\newcommand{\N}{\mathbb{N}}
\newcommand{\F}{\mathbb{F}}
\newcommand{\Q}{\mathbb{Q}}

\DeclareMathOperator{\G}{G}
\DeclareMathOperator{\Gal}{Gal}
\DeclareMathOperator{\Br}{Br}
\DeclareMathOperator{\Hom}{Hom}
\DeclareMathOperator{\cd}{cd}

\newcommand{\pow}[1]{(\!(#1)\!)}
\newcommand{\pair}[2]{\left[{#1}, {#2}\right)}
\newcommand{\KM}{\operatorname{K}^M}

\theoremstyle{definition}
\newtheorem{theorem}{Theorem}[section]
\newtheorem{lemma}[theorem]{Lemma}
\newtheorem{proposition}[theorem]{Proposition}
\newtheorem{corollary}[theorem]{Corollary}

\theoremstyle{remark}
\newtheorem{remark}[theorem]{Remark}

\author{Philip Dittmann}
\address{Institut für Algebra, Technische Universität Dresden, 01062 Dresden, Germany}
\curraddr{Department of Mathematics, University of Manchester, Manchester M13 9PL, United Kingdom}
\email{philip.dittmann@manchester.ac.uk}

\title[Characterising local fields by Galois theory and the Brauer group]{Characterising local fields of positive characteristic by Galois theory and the Brauer group}

\begin{document}

\begin{abstract}
  We show that each local field $\F_q\pow{t}$ of characteristic
  $p > 0$ is characterised up to isomorphism within the class of all
  fields of imperfect exponent at most $1$ by (certain small quotients
  of) its absolute Galois group together with natural axioms
  concerning the $p$-torsion of its Brauer group.  This complements
  previous work by Efrat--Fesenko, who analysed fields whose absolute
  Galois group is isomorphic to that of a local field of
  characteristic $p$.
\end{abstract}

\maketitle

\section{Introduction}

The birational part of Grothendieck's anabelian geometry,
and its close relative known as Bogomolov's programme,
are concerned with recovering fields
(say finitely generated over a prime field or an algebraically closed field)
from their absolute Galois groups or certain quotients thereof.
See for instance \cite[Chapter XII, §3]{NSW} and \cite[§2]{Szamuely_BourbakiPop} for an overview.
An important step in most approaches to such reconstruction questions
consists in first developing a \emph{local theory},
i.e.\ recovering the decomposition groups of certain valuations
from the given Galois group,
see for instance \cite[§3]{Szamuely_BourbakiPop}
(commenting on \cite{Pop_GrothendiecksConjectureBirationalAnabelian}) and
\cite{Pop_BirationalAnabelianProgramBogomolovI} (also for its historical notes).
An illustration of this already appears in \cite{Neukirch_KennzeichnungPAdischEndlichAlgebraisch} --
preceding the formulation of the conjectures of anabelian geometry --
concerning the characterisation of number fields by their absolute Galois groups:
Here a Galois-theoretic characterisation of algebraic $p$-adic numbers
and their finite extensions is given.
(See also the presentation in \cite[Chapter XII, §§1--2]{NSW}.)

By now there exists a substantial body of work concerning
valuation-theoretic consequences of Galois-theoretic input data;
see for instance the textbook presentations
\cite[§5.4]{EnglerPrestel} and \cite[Chapter 26]{Efrat_ValuationsOrderingsMilnorKTheory}
(both ultimately relying on the technique of rigid elements going back to \cite{Ware_rigid, ArasonElmanJacob_rigid})
and the literature referenced there.
One of the arguably most satisfying results in this regard is the complete
characterisation of the class of $p$-adically closed fields
(a well-behaved class containing the $p$-adic numbers as well as their
finite extensions and their algebraic parts)
in terms of their Galois group \cite{Pop_GaloisscheKennzeichnung, Efrat_GaloisCharacterization, Koenigsmann_rigidElementsValuationsGaloisCharacterization}
in a generalisation of \cite[Theorem 1]{Neukirch_KennzeichnungPAdischEndlichAlgebraisch},
and in strong analogy to the Galois-theoretic characterisation
of real-closed fields \cite{ArtinSchreier_KennzeichnungReellAbgeschlossen}.
This result on $p$-adically closed fields in turn has consequences
for the section conjecture in anabelian geometry \cite{Koenigsmann_SectionConjecture}.

In positive characteristic, there is no comparable characterisation of local fields.
Efrat and Fesenko in \cite{EfratFesenko} analysed the structure of fields whose absolute Galois group is isomorphic (as a profinite group) to the absolute Galois group of a local field $\F_q\pow{t}$.
The characterization of such fields was previously posed as \cite[Problem~4.6]{Efrat_GaloisCharacterization}.
These always carry a non-trivial henselian valuation,
yet Efrat--Fesenko show by examples that a wide variety of fields occurs:
this includes examples of characteristic $0$,
ones of arbitrarily large imperfection degree,
and ones with imperfect residue field.

In the present article we show that once information concerning the Brauer group of $K$,
or rather its $p$-torsion subgroup,
is added to the Galois-theoretic data,
the pathologies of \cite{EfratFesenko} disappear,
and in fact a characterisation of local fields $\F_q\pow{t}$
\emph{up to isomorphism} is obtained,
i.e.\ a characterisation much stronger than in the $p$-adic setting
(where $p$-adic local fields cannot be distinguished Galois-theoretically from
the larger class of $p$-adically closed fields).

\begin{theorem}\label{thm:intro-basic}
  Let $p$ be a prime number, $q = p^n$ for some $n \geq 1$.
  Let $K$ be a field satisfying the following axioms:
  \begin{enumerate}
    \item[(Gal)] The absolute Galois group $\G_K = \Gal(K^{\mathrm{sep}}/K)$ is isomorphic to $\G_{\F_q\pow{t}}$;
    \item[(Imp)] $K$ has exponent of imperfection at most $1$;
    \item[(Brau)] the $p$-torsion part $\Br(K)[p]$ of the Brauer group of $K$ is isomorphic to $\Z/p$;
    \item[(Pair)] the natural pairing $\Hom(\G_K, \Z/p) \times K^\times/K^{\times p} \to \Br(K)[p]$ (see Section \ref{sec:explaining}) induces an isomorphism $K^\times/K^{\times p} \cong \Hom(\Hom(\G_K, \Z/p), \Br(K)[p])$.
  \end{enumerate}
  Then $K$ is isomorphic to the local field $\F_q\pow{t}$.
\end{theorem}
It is a consequence of local class field theory that conversely the local field $\F_q\pow{t}$ satisfies the given axioms, see Proposition \ref{prop:local-satisfies-axioms}.
% kommentieren, dass diese Axiome anders als bei R und Q_p nicht erststufiger Natur sind?
We also give a strengthened version of Theorem \ref{thm:intro-basic} requiring not an isomorphism of full absolute Galois groups, but only of certain quotients, see Theorem \ref{thm:main-sharpened};
this is in the spirit of further developments of the anabelian and
valuation-theoretic results mentioned above, see
\cite{Pop_modpMetabelianBirationalSectionConjecture},
\cite{Topaz_CommutingLiftableII},
\cite{EfratMinac_SmallGaloisEncodeValuations}
for samples of a much wider literature.

We are not concerned here with reconstructing the local field $\F_q\pow{t}$ \emph{functorially} from its absolute Galois group,
in the sense of identifying the group of field automorphisms with the outer automorphism group of the absolute Galois group.
This was already done in \cite{Abrashkin_ModifiedLocalAnalogueGrothendieck}
(following \cite{Mochizuki_VersionGrothendieckPAdic} in the $p$-adic case),
after endowing the absolute Galois group with the ramification filtration.

Let us comment on why the axioms considered in Theorem \ref{thm:intro-basic} are natural.
In characteristic $p > 0$, absolute Galois groups fail to give much information ``at the prime $p$''; this is most plainly visible in Galois cohomology, since the $p$-Sylow subgroup of such absolute Galois groups are projective \cite[Corollary 6.1.3]{NSW}.
There is an alternative definition of cohomology groups $H^i(K, \Z/p(j))$ -- which in characteristic away from $p$ we take as a synonym for the Galois cohomology $H^i(\G_K, \mu_p^{\otimes j})$ -- which remedies this defect, see \cite[p.~143]{Kato_HassePrinciple}.
For instance, with this definition we have $H^2(K, \Z/p(1)) = \Br(K)[p]$ and $H^1(K, \Z/p(1)) = K^\times/p$ for all fields $K$, whether of characteristic $p$ or not.
With this notation, Theorem \ref{thm:intro-basic} can be rewritten more suggestively as follows:
\begin{theorem}\label{thm:intro-katoesque}
  Let $p$ be a prime number, $q = p^n$ for some $n \geq 1$.
  Let $K$ be a field satisfying the following axioms:
  \begin{enumerate}
  \item[(Gal)] $\G_K \cong \G_{\F_q\pow{t}}$
  \item[(Imp')] $H^2(K, \Z/p(2)) = 0$
  \item[(Brau')] $H^2(K, \Z/p(1)) \cong \Z/p\Z$
  \item[(Pair')] The cup-product pairing $H^1(K, \Z/p(0)) \times H^1(K, \Z/p(1)) \to H^2(K, \Z/p(1))$ induces an isomorphism $H^1(K, \Z/p(1)) \cong \Hom(H^1(K, \Z/p(0)), H^2(K, \Z/p(1)))$.
  \end{enumerate}
  Then $K$ is isomorphic to the local field $\F_q\pow{t}$.
\end{theorem}
This suggests that the additional axioms beyond $\G_K \cong \G_{\F_q\pow{t}}$ merely serve to fix the defect concerning information at $p$ of Galois cohomology in characteristic $p$.

The definition of cohomology groups $H^i(K, \Z/p(i))$ involves Milnor $K$-groups modulo $p$;
in particular, a further phrasing of the axioms asserts that the mod $p$ graded Milnor $K$-ring $\KM_{\ast}(K^{\mathrm{sep}})/p$ of the separable closure of $K$ together with its $\G_K$-action is isomorphic to the corresponding data for $\F_q\pow{t}$.
In this sense, our theorem is reminiscent of Topaz's work on the characterisation of fields up to isomorphism by their Milnor $K$-ring \cite{Topaz_algDependenceMilnorKTheory}.
However, our axioms are much closer to being purely Galois-theoretic since we only consider mod $p$ data.

The proof of Theorem \ref{thm:intro-basic} is in two steps:
Firstly we use Galois-theoretic information at primes $l \neq p$ to show that $K$ carries a non-trivial henselian valuation $v$.
This is standard and can be found in \cite{EfratFesenko}.
We recapitulate these results in Section \ref{sec:exists-valn}, where we also correct a mistake in \cite{EfratFesenko}.
Secondly we use the additional information given by the Brauer group to show that the valuation $v$ is discrete with finite residue field and $K$ is complete.
This is done in Section \ref{sec:pinning-down}, and is loosely inspired by Pop's use of pro-$p$ information in the absolute Galois group $\G_{\Q_p}$ in \cite[Theorem E.9]{Pop_GaloisscheKennzeichnung}.

The phenomenon found here that information away from the residue characteristic $p$ is much easier to use with the generally available tools,
but also yields much weaker results, than information at the prime $p$ seems quite typical.
See also \cite{KoenigsmannStrommen} in this regard.

\subsection*{Acknowledgements}

I would like to thank the anonymous referee for a very careful reading of the manuscript
and a number of helpful suggestions, which helped improve the presentation.

\section{Discussion of the axioms}
\label{sec:explaining}

Let $K$ be an arbitrary field and $p$ a prime number.
We briefly recall some basic properties of the Brauer group, and define the pairing function
\[ \Hom(\G_K, \Z/p) \times K^\times/p \to \Br(K)[p] \]
occurring in our axioms.
Here and throughout we write $\Hom(\G_K, \Z/p)$ for the (discrete) abelian group of \emph{continuous} group homomorphisms from $\G_K$ to the discrete group $\Z/p$,
$K^\times/p$ for the group $K^\times/K^{\times p}$ of $p$-th power classes,
and $\Br(K)[p]$ for the set of elements of $\Br(K)$ of order dividing $p$.

The \emph{Brauer group} $\Br(K)$ of $K$ can equivalently be defined as the Galois cohomology group $H^2(\G_K, {K^{\mathrm{sep}}}^\times)$, or as the collection of central simple algebras over $K$ modulo Brauer equivalence with the tensor product as a binary operation.
We will not need the second viewpoint, which is explored for instance in \cite[§2.4]{GilleSzamuely_2nd}; see \cite[Theorem 4.4.3]{GilleSzamuely_2nd} for the equivalence of the definitions.

Given a character $\chi \colon \G_K \to \Z/p$ and an element $a \in K^\times$, we can define an element $(\chi, a) \in \Br(K)$ as follows, see \cite[Chapitre XIV, §1]{Serre_corps-locaux} or \cite[Proposition 4.7.3]{GilleSzamuely_2nd}:
Observe that $\chi \in \Hom(\G_K, \Z/p) = H^1(\G_K, \Z/p)$, and we have an exact sequence $0 \to \Z \to \Z \to \Z/p \to 0$ of abelian groups with trivial $\G_K$-action, where the first map is multiplication by $p$.
Hence we can apply the associated boundary map $\delta$ in cohomology to obtain an element $\delta(\chi) \in H^2(\G_K, \Z)$.
Secondly, we have $a \in K^\times = H^0(\G_K, {K^{\mathrm{sep}}}^\times)$.
Taking the cup product of these two cohomology classes we obtain $(\chi, a) := \delta(\chi) \cup a \in H^2(\G_K, {K^{\mathrm{sep}}}^\times) = \Br(K)$.
In the central simple algebra viewpoint, $(\chi, a)$ is the class of the cyclic algebra associated to $\chi$ and $a$.
By construction, this gives a bilinear map of abelian groups $\Hom(\G_K, \Z/p) \times K^\times \to \Br(K)$.
Since $\Hom(\G_K, \Z/p)$ is a $p$-torsion group, it follows that the image must actually lie in $\Br(K)[p]$, and the second argument factors through $K^\times/p$.
Thus we have the desired bilinear pairing $\Hom(\G_K, \Z/p) \times K^\times/p \to \Br(K)[p]$.
This pairing induces a group homomorphism \[ K^\times \to \Hom(\Hom(\G_K, \Z/p), \Br(K)[p]),\]
where the right-hand side is the group of homomorphisms $\Hom(\G_K, \Z/p) \to \Br(K)[p]$ of discrete groups.

An essential property of the pairing $(\cdot, \cdot)$ is the following:
\begin{lemma}\label{lem:kernel-pairing-norm}
  Let $a \in K^\times$ and let $\chi \colon \G_K \to \Z/p$ be a non-zero homomorphism.
  Let $L/K$ be the $\Z/p$-extension associated to the index-$p$ normal subgroup $\ker(\chi)$ of $\G_K$.
  Then $(\chi, a) = 0$ in $\Br(K)[p]$ if and only if $a$ is a norm of the field extension $L/K$.
\end{lemma}
\begin{proof}
  See \cite[Corollary 4.7.5]{GilleSzamuely_2nd}.
\end{proof}

Let us show that local fields $\F_q\pow{t}$ satisfy the axioms of Theorem \ref{thm:intro-basic}.
\begin{proposition}\label{prop:local-satisfies-axioms}
  Suppose $K$ is a non-archimedean local field of arbitrary characteristic,
  for instance $K = \F_q\pow{t}$ for some prime power $q$ (not necessarily a power of $p$).
  Then $\Br(K)[p] \cong \Z/p$, and the pairing above induces an isomorphism
  \[ K^\times/p \cong \Hom(\Hom(\G_K, \Z/p), \Br(K)[p]). \]
\end{proposition}
\begin{proof}
  This is a consequence of local class field theory.
  Firstly, there is a canonical isomorphism $\operatorname{inv}_K \colon \Br(K) \xrightarrow{\sim} \Q/\Z$, and thus $\Br(K)[p] \cong \frac{1}{p}\Z/\Z \cong \Z/p$ \cite[Corollary 7.1.6]{NSW}.
  %The crux now is that the pairing $\Hom(\G_K, \Z/p) \times K^\times/p \to \Br(K)[p] \cong \Z/p$ is a perfect pairing of locally compact topological groups, i.e.\ it represents each of the discrete group $\Hom(\G_K, \Z/p)$ and the compact group $K^\times/p$ as the Pontryagin dual of the other.
  
  Secondly, class field theory provides the reciprocity homomorphism \[ (\cdot, K) \colon K^\times \to \G_K/[\G_K, \G_K] = \G_K^{\mathrm{ab}},\] which fits into the following exact sequence of abelian groups \cite[Theorem 7.2.11]{NSW} (cf.~also \cite[Chapitre XIII, §4, Remarque after Proposition 13]{Serre_corps-locaux}):
  \[ 0 \to K^\times \to \G_K^{\mathrm{ab}} \to \hat\Z/\Z \to 0 \]
  % Referenz für \hat\Z/\Z torsionsfrei und teilbar wäre schön
  Since the group $\hat\Z/\Z$ is torsion-free and divisible, a simple diagram chase shows that the induced homomorphism
  $K^\times/p \to \G_K^{\mathrm{ab}}/p$ is an isomorphism.
  (Alternatively, take the tensor product of the sequence above with the abelian group $\Z/p$, and use that $\operatorname{Tor}^\Z_1(\hat\Z/\Z, \Z/p)$ and $(\hat\Z/\Z)/p$ both vanish, again since $\hat\Z/\Z$ is torsion-free and divisible.)
  By Pontryagin duality, we further have a canonical isomorphism $\G_K^{\mathrm{ab}}/p \cong \Hom(\Hom(\G_K^{\mathrm{ab}}/p, \Z/p), \Z/p) = \Hom(\Hom(G_K, \Z/p), \Z/p)$.

  We obtain a diagram
  \[ \xymatrix{
      K^\times/p \ar[d] \ar[r] & \Hom(\Hom(\G_K, \Z/p), \Z/p) \\
      \Hom(\Hom(\G_K, \Z/p), \Br(K)[p]) \ar[ur] &
    }\]
  where the left arrow is induced by the pairing, the top arrow is induced by the reciprocity map and Pontryagin duality, and the diagonal is induced by the canonical isomorphism $\operatorname{inv}_K$.
  The diagram commutes, see \cite[Chapitre XIV, §1, Proposition 3]{Serre_corps-locaux} or \cite[Proposition 7.2.12]{NSW}.
  Since the top and the diagonal arrow are isomorphisms, the same must be true for the left arrow.
\end{proof}

A basic observation concerning fields satisfying the axioms of Theorem \ref{thm:intro-basic} is the following.
\begin{lemma}\label{lem:char-p}
  Suppose $K$ satisfies $\Br(K)[p] \neq 0$ and $\G_K$ has $p$-cohomological dimension at most $1$.
  Then $K$ has characteristic $p$.
\end{lemma}
\begin{proof}
  If $K$ did not have characteristic $p$, we would have $\Br(K)[p] = H^2(\G_K, \mu_p)$ \cite[Corollary 4.4.5]{GilleSzamuely_2nd}, but the left-hand side is non-zero, and $\G_K$ has $p$-cohomological dimension at most $1$.
\end{proof}
The hypothesis on $p$-cohomological dimension is satisfied by the absolute Galois group $\G_{\F_q\pow{t}}$,
as indeed it is by the absolute Galois group of any field of characteristic $p$ \cite[Corollary 6.1.3]{NSW}.
We later give a strengthened version of Lemma \ref{lem:char-p}, only requiring Galois-theoretic information on a certain quotient of $\G_K$.
Since this is more complicated to prove, and not needed for Theorem \ref{thm:intro-basic},
we defer it to Lemma \ref{lem:char-p-sharpened}.

Let us now discuss the rephrasing of Theorem \ref{thm:intro-basic} as Theorem \ref{thm:intro-katoesque}, using Kato's cohomology groups $H^i(K, \Z/p(j))$.
For $K$ of characteristic not $p$, these groups are defined as Galois cohomology groups $H^i(\G_K, \mu_p^{\otimes j})$ and in particular $H^2(K, \Z/p(1)) \cong \Br(K)[p]$ \cite[p.~143]{Kato_HassePrinciple}.
By Lemma \ref{lem:char-p}, we may therefore assume for the purposes of comparing the axioms of Theorem \ref{thm:intro-basic} and Theorem \ref{thm:intro-katoesque} that $K$ has characteristic $p$.

We may then take it as the definition of these groups that $H^i(K, \Z/p(i)) = \KM_i(K)/p$ and $H^{i+1}(K, \Z/p(i)) = H^1(\G_K, \KM_i(K^{\mathrm{sep}})/p)$,
where the last group is a Galois cohomology group, and $\KM_\ast$ denotes the graded Milnor $K$-ring.
(See \cite[Appendix A, p.~152]{GaribaldiMerkurjevSerre_RostInvariants} and the Bloch--Kato--Gabber Theorem, cf.\ \cite[p.~149, (1.2.1)]{Kato_HassePrinciple}.)

This means we have $H^1(K, \Z/p(1)) = \KM_1(K)/p = K^\times/p$,
$H^1(K, \Z/p(0)) = H^1(\G_K, \Z/p) = \Hom(\G_K, \Z/p)$, and $H^2(K, \Z/p(1)) = H^1(\G_K, {K^{\mathrm{sep}}}^\times/p) \cong \Br(K)[p]$ (canonically) \cite[Appendix A, Example A.3~(2)]{GaribaldiMerkurjevSerre_RostInvariants}.
Since $H^1(K, \Z/p(1)) = K^\times/p$ embeds into the Galois cohomology group $H^0(\G_K, {K^{\mathrm{sep}}}^\times/p)$,
the cup product in Galois cohomology yields a bilinear pairing
\[ H^1(K, \Z/p(0)) \times H^1(K, \Z/p(1)) \to H^2(K, \Z/p(1)) ,\]
cf.\ \cite[Appendix A, (A.5)]{GaribaldiMerkurjevSerre_RostInvariants}.
Standard facts
(namely the compatibility of the cup product with the connecting homomorphisms in long exact sequences)
show that this pairing agrees up to sign with the pairing into $\Br(K)[p]$ defined above.
This gives the equivalence of axioms (Pair) and (Pair').

To establish the equivalence of Theorem \ref{thm:intro-basic} and Theorem \ref{thm:intro-katoesque},
we also need to translate the axiom that $H^2(K, \Z/p(2)) = \KM_2(K)/p$ vanishes into the statement that $[K : K^p] \leq p$.
Consider two elements $a, b \in K^\times$.
These are $p$-dependent, meaning that $[K^p(a,b) : K^p] \leq p$,
if and only if the elements $da$ and $db$ of the vector space of absolute Kähler differentials $\Omega_K$ are linearly dependent,
which is the case if and only if the differential form $da \wedge db \in \Omega^2_K$ vanishes,
cf.\ \cite[Proposition~A.8.11]{GilleSzamuely_2nd} or \cite[Chapitre V, No 2, §13, Théorème 1~a)]{Bourbaki_algebre-4567}.
By the Bloch--Kato--Gaber Theorem \cite[Theorem~9.5.2]{GilleSzamuely_2nd},
the group $\KM_2(K)/p$ naturally embeds into $\Omega^2_K$, from which we deduce
that $da \wedge db \in \Omega^2_K$ vanishes if and only if the symbol $\{ a, b \} \in \KM_2(K)/p$ vanishes.
This shows that the group $\KM_2(K)/p$ vanishes if and only if any two $a, b \in K^\times$ are $p$-dependent,
i.e.\ if and only if $[K : K^p] \leq p$.

\section{Existence of a valuation}
\label{sec:exists-valn}

Let $K$ be a field with $\G_K \cong \G_{\F_q\pow{t}}$.
The goal of this section is to show that $K$ carries a non-trivial henselian valuation $v$ whose residue field $Kv$ is not $p$-closed,
i.e.\ such that $Kv$ has a $\Z/p$-extension or equivalently the maximal pro-$p$ quotient $\G_{Kv}(p)$ is not trivial.

\begin{remark}\label{rem:EF-mistake}
This result follows in principle from \cite[Proposition 4.3~(d)]{EfratFesenko}.
However, the proof of this proposition relies on \cite[Lemma 4.2]{EfratFesenko}, a lemma concerning profinite groups, which is false.
The lemma in particular claims that for a profinite group $H$ such that
the $l$-Sylow subgroups of $H$ are isomorphic to $\Z_l$
for every prime number $l$,
the maximal prime-to-$p$-quotient $H(p')$ is isomorphic to $\hat\Z/\Z_p$
for every $p$.

To see that this is false,
consider the profinite group $H = \Z_2 \ltimes \prod_{l \neq 2} \Z_l$, where the product is over odd primes $l$, and $x \in \Z_2$ acts on $\prod_{l \neq 2} \Z_l$ as multiplication by $(-1)^x$.
Then every Sylow subgroup of $H$ is procyclic.
However, for $p \neq 2$, the maximal prime-to-$p$ quotient $H(p')$ is the non-abelian $\Z_2 \ltimes \prod_{l \neq 2, p} \Z_l$, which contradicts \cite[Lemma 4.2]{EfratFesenko}.
One can also check that the maximal prime-to-$2$ quotient $H(2')$ is the trivial group for a further contradiction.
The mistake in the proof of \cite[Lemma 4.2]{EfratFesenko} lies in an erroneous statement concerning the canonical projection $\hat{F} \to \hat{F}(\mathrm{ab}, p')$ in the notation there, where $\hat{F}$ is a free profinite group.

Although parts of the proof of \cite[Proposition 4.3]{EfratFesenko} could be salvaged, we use this opportunity to give a new full proof of \cite[Proposition 4.3]{EfratFesenko}, or rather of a slightly strengthened formulation.
\end{remark}

\begin{proposition}\label{prop:EF-replacement}
  There exists a non-trivial henselian valuation $v$ on $K$ such that $Kv$ has characteristic $p$, $(vK : l\cdot vK) = l$ for every prime number $l \neq p$,
  the maximal prime-to-$p$ quotient $\G_{Kv}(p')$ is isomorphic to $\hat{\Z}/\Z_p$,
  and $Kv \cap \overline{\F_p} \cong \F_q$.
\end{proposition}
\begin{proof}
  The existence of a henselian valuation $v$ with $(vK : l \cdot vK) = l$ for every $l \neq p$ and residue characteristic $p$ is \cite[Theorem 4.1]{EfratFesenko}.

  (We mention in passing that this also follows very neatly from Koenigsmann's characterisation of ``tamely branching valuations'' as presented in \cite[Theorem 5.4.3]{EnglerPrestel}:
  Indeed, by the theorem the canonical henselian valuation $v$ on $K$ must be tamely branching at all $l \neq p$ since the same is true for $\F_q\pow{t}$ and this is a Galois-theoretic condition;
  the residue field $Kv$ must have characteristic $p$ by the argument of \cite[Theorem 4.1~(b)]{EfratFesenko};
  and since the abelian normal subgroups of $l$-Sylow subgroups of $\G_{\F_q\pow{t}}$ are procyclic, it follows that $[vK : l\cdot vK] = l$ from ramification theory.)

  It remains to prove the assertions about $\G_{Kv}(p')$ and $Kv \cap \overline{\F_p}$.
  We extensively use the theory of ramification and inertia subgroups as presented for instance in \cite[Chapter 15]{Efrat_ValuationsOrderingsMilnorKTheory}.
  Let $\sigma \colon \G_{\F_q\pow{t}} \to \G_K$ be an isomorphism, and write $T$, $V$ for the inertia and ramification subgroups of $\G_{\F_q\pow{t}}$.
  Let $L/K$ be the fixed field of $\sigma(V) \leq \G_K$.
  Then $L/K$ is a Galois extension with Galois group isomorphic to $\G_{\F_q\pow{t}}/V$, so it is generated as a profinite group by two elements $\sigma$, $\tau$ with the sole relation $\sigma\tau\sigma^{-1} = \tau^q$, i.e.\ $\Gal(L/K)$ is the semi-direct product of $\langle\sigma\rangle \cong \hat\Z$ and $\langle\tau\rangle \cong \hat\Z/\Z_p$ \cite[Theorem 7.5.3]{NSW}.
  Furthermore, the absolute Galois group of $L$ is isomorphic to $V$, and thus a free pro-$p$ group.
  In particular, for every prime number $l \neq p$ the value group $vL$ (where we also write $v$ for the unique prolongation of the valuation $v$ on $K$ to $L$) is $l$-divisible.

  The relative ramification subgroup $R$ of $\Gal(L/K)$ is a normal pro-$p$-subgroup.
  The $p$-Sylow subgroups of $\Gal(L/K) = \langle \sigma, \tau\rangle$ are procyclic, so $R$ is abelian.
  % Der Beweis dieses Korollars in NSW ist _sehr_ sparsam, aber ich habe mich von der Richtigkeit der Aussage überzeugt.
  By \cite[Corollary 7.5.7~(ii)]{NSW}, $R$ must be contained in $\langle\tau\rangle \cong \hat\Z/\Z_p$, which has no pro-$p$-part.
  Therefore $R$ is trivial, i.e.\ the extension $L/K$ is tamely ramified.
  
  Thus the relative inertia group $I$ of $L/K$ is an abelian normal subgroup of $\Gal(L/K) = \langle \sigma, \tau \rangle$.
  By \cite[Corollary 7.5.7~(ii)]{NSW} once more, $I \leq \langle \tau\rangle \cong \hat\Z / \Z_p$.
  We have the so-called ramification pairing
  \[ I \times vL/vK \to \mu(Lv) \]
  where $\mu(Lv)$ is the group of roots of unity in the residue field $\mu(Lv)$,
  see \cite[Theorem 16.2.6]{Efrat_ValuationsOrderingsMilnorKTheory}.
  (In the notation there, observe that $L$ is the ramification subfield of $L/K$ with respect to $v$ since
  $L/K$ is tamely ramified, and the value group $vK$ agrees with the value group of the relative inertia field
  since $L/K$ is separable, cf.\ the diagram on \cite[p.~149]{Efrat_ValuationsOrderingsMilnorKTheory}.)
  The ramification pairing is bilinear with trivial left and right kernel,
  and it is continuous (where $vL/vK$ and $\mu(Lv)$ carry the discrete topology and $I$ its profinite topology).
  Since $vL$ is $l$-divisible for every prime number $l \neq p$, but $vK$ is not, the torsion group
  $vL/vK$ contains elements of arbitrary $l$-power order.
  The triviality of the right kernel of the pairing then implies that $\mu(Lv)$ contains elements of arbitrary $l$-power order,
  and furthermore that the profinite group $I$ must have non-trivial $l$-Sylow subgroup.
  It follows that Sylow subgroups of the quotient $\langle\tau\rangle/I$ are torsion, since the $l$-Sylow subgroup of both $\langle\tau\rangle$ and $I$ is isomorphic to $\Z_l$ for every $l \neq p$.

  We have $\Gal(Lv/Kv) \cong \langle\sigma,\tau\rangle/I$.
  In particular the $p$-Sylow subgroups of $\Gal(Lv/Kv)$ are isomorphic to $\Z_p$, so $\cd_p \Gal(Lv/Kv) = 1$.
  Since $\G_{Lv}$ is a pro-$p$-group as a quotient of $\G_L$, it follows from \cite[Theorem 3.5.6]{NSW} that the exact sequence
  \[ 1 \to \G_{Lv} \to \G_{Kv} \to \Gal(Lv/Kv) \to 1 \]
  splits.
  Hence $\Gal(Lv/Kv)$ is isomorphic to a subgroup of $\G_{Kv}$ and therefore torsion-free.
  Therefore the subgroup $\langle\tau\rangle/I$ of $\langle\sigma,\tau\rangle/I \cong \Gal(Lv/Kv)$ is also torsion-free.
  On the other hand, we have seen above that its Sylow subgroups are torsion,
  so we must have $I = \langle\tau\rangle$.
  Thus $\Gal(Lv/Kv) \cong \langle\sigma,\tau\rangle/I \cong \hat\Z$.
  Since $\G_{Lv}$ is pro-$p$, we deduce $\G_{Kv}(p') = \Gal(Lv/Kv)(p') = \hat\Z/\Z_p$.

  It remains to prove that $Kv \cap \overline{\F_p} \cong \F_q$.
  We saw above in our analysis of $\mu(Lv)$ that $Lv$ contains $l^s$-roots of unity for every prime number $l \neq p$ and every $s \geq 1$.
  In other words, $Lv$ contains the algebraic closure $\overline{\F_p}$.

  The ramification pairing induces an isomorphism $I \cong \Hom(vL/vK, \mu(Lv))$ \cite[Corollary~16.2.7~(c)]{Efrat_ValuationsOrderingsMilnorKTheory},
  where the homomorphism group is endowed with the topology of pointwise convergence.
  This isomorphism is compatible with the action of $\langle\sigma\rangle \cong \Gal(Lv/Kv)$,
  where $\Gal(Lv/Kv)$ acts on the homomorphism group via its action on $\mu(Lv)$, and $\langle\sigma\rangle$ acts on $I$ by conjugation.

  The automorphism group $\operatorname{Aut}(\overline{\F_p}) = \Gal(\overline{\F_p}/\F_p)$ is isomorphic to $\hat\Z$ and generated
  by the Frobenius automorphism $x \mapsto x^p$.
  Say the image of $\sigma$ in $\Gal(Lv/Kv)$ acts on $\overline{\F_p}$ as $x \mapsto x^{p^n}$ for some $n \in \hat\Z$.
  Therefore this element of $\Gal(Lv/Kv)$ acts on the abelian profinite group $\Hom(vL/vK, \mu(Lv))$ as multiplication by $p^n$.
  Compatibility of the action with the isomorphism induced by the ramification pairing means that
  $\sigma$ must act in the same way on $I$ by conjugation.
  In other words, we have $\sigma\eta\sigma^{-1} = \eta^{p^n}$ for each $\eta \in I = \langle\tau\rangle$.
  However, by construction of $\langle\sigma,\tau\rangle$ as a semi-direct product we know that $\sigma\eta\sigma^{-1} = \eta^q$.
  This is only possible for all elements $\eta \in I = \langle\tau\rangle \cong \hat\Z/\Z_p$ if we have $p^n = q$ as elements of $(\hat\Z/\Z_p)^\times$.
  It follows that an element $x \in \overline{\F_p}$ lies in $Kv$, i.e.~is fixed by $\Gal(Lv/Kv)$,
  if and only if $x^{p^n} = x^q = x$.
  Therefore $\overline{\F_p} \cap Kv = \F_q$.
\end{proof}

Proposition \ref{prop:EF-replacement} recovers \cite[Proposition 4.3]{EfratFesenko}, and also immediately implies the following desired corollary.
\begin{corollary}\label{cor:val-exists}
  There exists a non-trivial henselian valuation $v$ on $K$ such that the residue field $Kv$ is not $p$-closed.
\end{corollary}
\begin{proof}
  Let $v$ be the valuation from Proposition \ref{prop:EF-replacement}.
  Since $\F_q$ has a $\Z/p$-extension, so does $Kv$.
\end{proof}

We later give an alternative version of Corollary \ref{cor:val-exists},
requiring only weaker Galois-theoretic information,
but also only yielding a $p$-henselian valuation.
Since this is not needed for the basic result Theorem \ref{thm:intro-basic},
and the proof uses different techniques as it cannot rely on \cite[Theorem 4.1]{EfratFesenko},
we defer this to Proposition \ref{prop:val-exists-sharpened}.

\section{Consequences of the Brauer group axioms}
\label{sec:pinning-down}

Throughout this section, let $K$ be a field such that:
\begin{enumerate}
\item $K$ has countably many $\Z/p$-extensions;
\item $K$ has characteristic $p$;
\item the axioms (Brau) and (Pair) from Theorem \ref{thm:intro-basic} are satisfied.
\end{enumerate}

Artin--Schreier theory yields an isomorphism $\Hom(\G_K, \Z/p) \cong K/\wp(K)$,
where $\wp(K)$ is the image of the group homomorphism $\wp \colon K \to K, x \mapsto x^p - x$ \cite[Corollary 6.1.2]{NSW}.
Our first assumption can therefore also be phrased as asserting that the group $K/\wp(K)$ is countable.
The Artin--Schreier isomorphism yields a pairing $\pair{\cdot}{\cdot} \colon K/\wp(K) \times K^\times/p \to \Br(K)[p]$
from the pairing defined in Section \ref{sec:explaining}.
(This notation is also used in \cite[Chapitre XIV, §5]{Serre_corps-locaux}.)
We will also consider $\pair{\cdot}{\cdot}$ as a function defined on $K \times K^\times$.

In this section we show that certain valuations $v$ on $K$ satisfy very strong properties.
In the proof of Theorem~\ref{thm:intro-basic},
we will apply this to a valuation obtained
(under the hypothesis (Gal), beyond the standing assumptions of this section introduced above)
from the results of Section~\ref{sec:exists-valn}.
However, for the time being we develop the material
for arbitrary valuations $v$, of which we will generally only assume that they are non-trivial,
have non-$p$-closed residue field $Kv$,
and are \emph{$p$-henselian}.
This last standard condition, see \cite{ChatzidakisPerera_pHenselianity},
means that the maximal ideal $\mathfrak{m}_v$ of the valuation ring is contained in $\wp(K)$,
or equivalently that the valuation extends uniquely to every $\Z/p$-extension of $K$.
(We will meet the general condition of $l$-henselianity,
where $l$ is a prime not necessarily equal to the characteristic,
in Section~\ref{sec:sharpenings-quotients} below.)

\begin{lemma}\label{lem:pairing-props}
  For a $p$-henselian valuation $v$ on $K$, we have $\pair{\mathcal{O}_v}{1 + \mathfrak{m}_v} = 0$.
  More generally:
  For every $0 \neq x \in \mathcal{O}_v$, we have $\pair{x^{-1}\mathcal{O}_v}{1 + x\mathfrak{m}_v} = 0$.
\end{lemma}
\begin{proof}
  Let $a \in x^{-1}\mathcal{O}_v$, i.e.\ $v(a) \geq -v(x)$, and consider the character $\chi \colon \G_K \to \Z/p$ associated to $a + \wp(K)$.
  Let $b \in 1 + x\mathfrak{m}_v$.
  We must show that $(\chi, b) = 0$.
  There is nothing to prove if $\chi = 0$, so suppose otherwise.
  Let $L/K$ be the $\Z/p$-extension associated to $\chi$, which is precisely the extension of $K$ generated by an element $\alpha$ with $\alpha^p - \alpha = a$.
  By Lemma \ref{lem:kernel-pairing-norm} we have to show that $b$ is a norm of $L/K$.
  For an element $y \in K$, we have $(\alpha+y)^p - (\alpha+y) - a - \wp(y) = 0$, so $X^p - X - a - \wp(y)$ is the minimal polynomial of $\alpha+y$ over $K$, and hence the norm of $\alpha+y$ is $a + \wp(y)$.
  The norm of $\frac{\alpha+y}{\alpha}$ is therefore $\frac{a + \wp(y)}{a} = 1 + \wp(y)/a$.
  Since $v$ is $p$-henselian,
  we can choose $y \in K$ with $\wp(y) = (b-1)a = \frac{b-1}{x} (ax) \in \mathfrak{m}_v$.
  Then $b = 1 + \wp(y)/a$ is a norm, as desired.
\end{proof}

\begin{lemma}\label{lem:imperfect-implies-uniformiser}
  Let $v$ be a non-trivial $p$-henselian valuation on $K$.
  If $Kv$ is imperfect, then the value group $vK$ is countable and has a minimal positive element.
\end{lemma}
\begin{proof}
  Let us first argue that the field $Kv$ is countable.
  Pick $x \in K$ with $v(x) < 0$ and consider the $\F_p$-linear map
  \[ Kv/(Kv)^p \to K/(\wp(K) + x^p \mathfrak{m}_v), (y + \mathfrak{m}_v) + (Kv)^p \mapsto yx^p + \wp(K) + x^p\mathfrak{m}_v \]
  (where $y \in \mathcal{O}_v$).
  This map is well-defined since if $y \in z^p + \mathfrak{m}_v$ for some $z \in \mathcal{O}_v$,
  we have \[ yx^p \in z^px^p + x^p \mathfrak{m}_v = \wp(zx) + x^p \frac{z}{x^{p-1}} + x^p \mathfrak{m}_v \subseteq \wp(K) + x^p \mathfrak{m}_v .\]
  This map is also injective:
  Indeed, if $y \in \mathcal{O}_v$ is such that $yx^p \in \wp(K) + x^p \mathfrak{m}_v$,
  we have $y \in \mathfrak{m}_v + x^{-p}\wp(K)$,
  so there exists $z \in K$ with $y - x^{-p}(z^p-z) \in \mathfrak{m}_v$.
  Comparing valuations, we must have $v(z) \geq v(x)$,
  and so $x^{-p}z \in \mathfrak{m}_v$,
  yielding $y - (z/x)^p \in \mathfrak{m}_v$.
  This shows that $y + \mathfrak{m}_v \in Kv$ is a $p$-th power, proving that the map has trivial kernel.
  Since the codomain of the map is a quotient of of $K/\wp(K)$ and
  therefore countable by our standing assumption,
  the positive-dimensional $(Kv)^p$-vector space $Kv/(Kv)^p$ must also be countable.
  This shows that $(Kv)^p$ and therefore $Kv$ are countable.

  Let $a \in \mathcal{O}_v$ such that the residue of $a$ in $Kv$ is not a $p$-th power.
  For every $\gamma \in vK$, $\gamma < 0$, choose an element $x_\gamma$ of valuation $\gamma$, and set $y_\gamma  = a x_\gamma^p$.
  We claim that the images of the $y_\gamma$ in $K/\wp(K)$ are $\F_p$-linearly independent.
  Indeed, if they were not, then there would exist $\gamma_1 < \gamma_2 < \dotsb \gamma_n < 0$ and coefficients $a_2, \dotsc, a_n \in \F_p$, as well as an element $z \in K$ with $y_{\gamma_1} = a_2 y_{\gamma_2} + \dotsb + a_n y_{\gamma_n} +
  z^p - z$.
  Considering the valuation of each term, we must have $v(z) = v(y_{\gamma_1})/p = v(x_{\gamma_1})$.
  Dividing by $x_{\gamma_1}^p$ and taking residues, we obtain $\overline{a} = \overline{z/x_{\gamma_1}}^p$, contradicting the choice of $a$.
  This shows the linear independence claimed.

  In particular, it follows that $vK$ is countable since $K/\wp(K)$ is countable.
  Assume now that $vK$ does not have a minimal positive element, so $vK$ is densely ordered.
  We may then fix a sequence $\delta_1, \delta_2, \dotsc$ of negative elements of $vK$ converging to $0$.
  There exists a family of functions $f_i \in \N \to \{0, 1\}$, indexed by $i$ in some uncountable index set $I$, such that for $i \neq j$ the functions $f_i$ and $f_j$ differ in infinitely many places.
  (For instance, for every set $P$ of prime numbers, choose a function $f$ which is $1$ at each power of a prime in $P$ and $0$ everywhere else.)
  By the pairing axiom (Pair), and using the linear independence previously proved, there exist elements $z_i \in K^\times$, indexed by $i \in I$, such that $\pair{y_{\delta_n}}{z_i}$ is a non-zero element of $\Br(K)[p]$ if and only if $f_i(n) = 1$.

  The group $K^\times/(K^{\times p}(1+\mathfrak{m}_v))$ is countable as a group extension of $K^\times/(K^{\times p} \mathcal{O}_v^\times) \cong vK/p$
  by $(K^{\times p} \mathcal{O}_v^\times)/(K^{\times p}(1+\mathfrak{m}_v)) \cong \mathcal{O}_v^\times/(\mathcal{O}_v^{\times p}(1+\mathfrak{m}_v)) \cong (Kv)^\times/p$,
  both of which are countable groups by the previous steps.
  Therefore there must exist two indices $i \neq j$ such that $z_i/z_j \in K^{\times p} (1 + \mathfrak{m}_v)$, so we may write $z_i/z_j = a^p (1 + b)$ for suitable $a \in K^\times$, $b \in \mathfrak{m}_v$.

  Since the $\delta_n$ converge to $0$, there exists an $n_0$ such that for all $n \geq n_0$ we have $\delta_n > -v(b)$, and therefore
  \[ \pair{x_{\delta_n}}{z_i/z_j} = \pair{x_{\delta_n}}{1+b} = 0 \]
  by Lemma \ref{lem:pairing-props} (set $x = x_{\delta_n}^{-1}$ there).
  This forces $f_i(n) = f_j(n)$ for all $n \geq n_0$, contradicting our choice of functions $f_i$.
\end{proof}

\begin{lemma}\label{lem:perfect-implies-uniformiser}
  Let $v$ be a non-trivial $p$-henselian valuation on $K$ whose residue field $Kv$ is not $p$-closed.
  Suppose that $Kv$ is perfect and $vK$ has no non-trivial $p$-divisible convex subgroup.
  Then $vK$ is countable and has a minimal positive element.
\end{lemma}
\begin{proof}
  Let us first prove that $vK$ is countable.
  Pick $\gamma \in vK$, $\gamma < 0$, which is not in $p \cdot vK$.
  Fix $x \in K$ with $vx = \gamma$.
  For every $\delta \in vK$, $\delta < 0$, choose $y_\delta \in K$ with $vy_\delta = \delta$, and set $x_\delta = x y_\delta^p$.
  Since the elements $x_\delta$ have distinct negative valuations not divisible by $p$,
  no non-trivial $\F_p$-linear combination of the $x_\delta$ can lie in $\wp(K)$.
  Therefore the elements $x_\delta + \wp(K)$ are $\F_p$-linearly independent in $K/\wp(K)$.
  Since the latter is countable by our standing assumption, this shows that $vK$ is countable.

  Suppose now that $vK$ has no minimal positive element, so that $vK$ is densely ordered.
  Since $vK$ has no non-trivial $p$-divisible convex subgroup,
  every open interval in $vK$ containing $0$ contains elements which are not $p$-divisible.
  We can therefore choose a sequence $\delta_1, \delta_2, \dotsc$ of negative elements of $vK$ which converges to $0$ and such that no $\delta_i$ is divisible by $p$.

  For every $n$, let $x_n \in K$ with $vx_n = \delta_n$.
  As above, the elements $x_n + \mathcal{O}_v + \wp(K) \in K/(\mathcal{O}_v + \wp(K))$ are $\F_p$-linearly independent.
  Consequently there exists an $\F_p$-linear map $f \colon K/\wp(K) \to \Br(K)[p]$
  with $f(\mathcal{O}_v) = 0$ and $f(x_n) \neq 0$ for each $n$.
  By the pairing axiom (Pair), there exists $z \in K^\times$ with $\pair{\cdot}{z} = f$,
  so that $\pair{\mathcal{O}_v}{z} = 0$ and $\pair{x_n}{z} \neq 0$ for each $n$.

  By assumption $Kv$ is not $p$-closed,
  so there exists an element $a \in \mathcal{O}_v$ whose residue does not lie in $\wp(Kv)$.
  This means that $a \not\in \wp(K)$,
  and so adjoining a root of $X^p-X-a$ to $K$ yields a $\Z/p$-extension $L/K$ by Artin--Schreier theory.
  The valuation $v$ extends uniquely to $L$ by $p$-henselianity,
  and the residue field of $L$ is a proper extension of the residue field $Kv$,
  since it contains an Artin--Schreier root of the residue of $a$, but by assumption $Kv$ does not.
  Therefore the value group does not change in the degree $p$ extension $L/K$ by \cite[Theorem~3.3.3]{EnglerPrestel}.
  Since $\pair{\mathcal{O}_v}{z} = 0$,
  Lemma \ref{lem:kernel-pairing-norm}
  (applied to the homomorphism $\chi \colon \G_K \to \Z/p$ corresponding to $a$ under
  the Artin--Schreier isomorphism $\Hom(\G_K,\Z/p) \cong K/\wp(K)$, whose kernel is $\G_L$)
  implies that $z$ is a norm of the extension $L/K$.
  It now follows from \cite[Remark 3.2.17]{EnglerPrestel} that
  the valuation of $z$ must be divisible by $p$.
  Hence we may as well suppose that $z \in \mathcal{O}_v^\times$ by multiplying $z$ by a $p$-th power.

  Since $Kv$ is perfect, we may further suppose that $z \in 1 + \mathfrak{m}_v$.
  However, Lemma \ref{lem:pairing-props} now implies that
  for all $n$ with $v(x_n) > -v(z-1)$ we have $\pair{x_n}{z} = 0$;
  since the values $v(x_n)$ converge to $0$, this is the case for all sufficiently large $n$, contradicting the choice of $z$.
\end{proof}

As usual, an element of a valued field whose valuation is minimal positive is called a \emph{uniformiser}.
\begin{lemma}\label{lem:finite-residue}
  Let $v$ be a non-trivial $p$-henselian valuation on $K$
  with a uniformiser $\pi$ and such that $Kv$ is not $p$-closed.
  Then the residue field $Kv$ is finite.
\end{lemma}
\begin{proof}
  By Lemma \ref{lem:imperfect-implies-uniformiser} and Lemma \ref{lem:perfect-implies-uniformiser}
  the value group $vK$ is countable.
  The embedding $\frac{1}{\pi}\mathcal{O}_v \hookrightarrow K$ induces an injection $\frac{1}{\pi}\mathcal{O}_v/\wp(\mathcal{O}_v) \hookrightarrow K/\wp(K)$
  of $\F_p$-vector spaces,
  which in turn induces a surjective linear map $\Hom(K/\wp(K), \Br(K)[p]) \to \Hom(\frac{1}{\pi}\mathcal{O}_v/\wp(\mathcal{O}_v), \Br(K)[p])$.
  We therefore obtain a surjective group homomorphism
  \[ K^\times/p \to \Hom(K/\wp(K), \Br(K)[p]) \to \Hom(\frac{1}{\pi}\mathcal{O}_v/\wp(\mathcal{O}_v), \Br(K)[p]) , \]
  where the first map is induced by the pairing function and surjective by the pairing axiom (Pair).
  The group $1 + \pi^2\mathcal{O}_v/(K^{\times p} \cap (1 + \pi^2\mathcal{O}_v))$ lies in the kernel by Lemma \ref{lem:pairing-props},
  so we obtain a surjective homomorphism
  \[ K^\times/(K^{\times p} (1 + \pi^2\mathcal{O}_v)) \to \Hom(\frac{1}{\pi}\mathcal{O}_v/\wp(\mathcal{O}_v), \Br(K)[p]). \]

  Suppose $Kv$ has infinite dimension $\kappa$ as an $\F_p$-vector space.
  Then $\frac{1}{\pi}\mathcal{O}_v/\wp(\mathcal{O}_v)$ also has dimension at least $\kappa$, since it has $\frac{1}{\pi}\mathcal{O}_v/\mathcal{O}_v \cong Kv$ as a quotient.
  Consequently $\Hom(\frac{1}{\pi}\mathcal{O}_v/\wp(\mathcal{O}_v), \Br(K)[p])$ has cardinality at least $2^\kappa$,
  as we can construct homomorphisms by freely choosing images of some fixed basis vectors.

  On the other hand, $K^\times/(K^{\times p} (1 + \pi^2 \mathcal{O}_v))$ has cardinality at most $\kappa$:
  Indeed, in the subgroup series
  \[ K^\times \ge K^{\times p} \mathcal{O}_v^\times \geq K^{\times p} (1 + \pi \mathcal{O}_v) \geq K^{\times p} (1 + \pi^2 \mathcal{O}_v) \]
  the first factor group $K^\times/(K^{\times p} \mathcal{O}_v^\times) \cong vK/p$ is countable,
  the second factor group \[ (K^{\times p} \mathcal{O}_v^\times) / (K^{\times p} (1 + \pi \mathcal{O}_v)) \cong \mathcal{O}_v^\times / (\mathcal{O}_v^{\times p} (1 + \pi \mathcal{O}_v)) \cong (Kv)^\times/p \] has cardinality at most $\kappa$,
  and the third factor group $(K^{\times p} (1 + \pi \mathcal{O}_v)) / (K^{\times p} (1 + \pi^2 \mathcal{O}_v))$ has cardinality at most $\kappa$
  since it has a surjection from $(1 + \pi \mathcal{O}_v)/(1 + \pi^2 \mathcal{O}_v) \cong Kv$.
  This shows that $K^\times/(K^{\times p} (1 + \pi^2 \mathcal{O}_v))$ has cardinality at most $\kappa$.
  Since we cannot have a surjective map from a set of cardinality at most $\kappa$ to one of cardinality at least $2^\kappa$,
  this is a contradiction.
\end{proof}

\begin{lemma}\label{lem:discrete}
  If $K$ carries a non-trivial $p$-henselian valuation $v$ with $Kv$ not $p$-closed,
  then $Kv$ is finite and $vK$ is isomorphic to $\Z$.
\end{lemma}
\begin{proof}
  If the residue field $Kv$ is not perfect,
  then by Lemma \ref{lem:imperfect-implies-uniformiser} there exists a uniformiser
  and by Lemma \ref{lem:finite-residue} the residue field $Kv$ is finite,
  and therefore perfect after all.

  Thus the residue field $Kv$ is perfect.
  If the value group $vK$ is $p$-divisible, every element of $K^\times$ can be written as
  the product of a $p$-th power with an element of $1 + \mathfrak{m}_v$.
  Lemma \ref{lem:pairing-props} implies that then $\pair{\mathcal{O}_v}{K^\times} = 0$.
  By the pairing axiom (Pair), this is only possible if $\mathcal{O}_v \subseteq \wp(K)$,
  implying that $Kv = \wp(Kv)$,
  but this is in contradiction to the assumption that $Kv$ is not $p$-closed.
  Therefore the value group $vK$ is not $p$-divisible.
  Let $w$ be the coarsening of $v$ corresponding to the largest $p$-divisible convex subgroup of $vK$.

  Since $Kv$ has a $\Z/p$-extension by assumption,
  there exists a $\Z/p$-extension of $K$ which is unramified with respect to $v$,
  in particular unramified with respect to $w$,
  so $Kw$ is not $p$-closed.
  Repeating the previous argument with $(K,w)$,
  we see that the residue field $Kw$ is perfect.
  By Lemma \ref{lem:perfect-implies-uniformiser}
  $(K,w)$ has a uniformiser $\pi$,
  and by Lemma \ref{lem:finite-residue} its residue field $Kw$ is finite.
  In particular, $v$ cannot be a proper refinement of $w$, so $w = v$.

  The subgroup of $vK$ generated by the value of the uniformiser $\pi$
  is isomorphic to $\Z$ and convex.
  Consider the coarsening $v'$ of $v$ with value group $v'K = vK/\langle v(\pi)\rangle$
  (i.e.\ the finest proper coarsening of $v$).
  The residue field $Kv'$ carries a valuation with value group $\Z$,
  and so $Kv'$ is imperfect.
  As above, $Kv'$ is also not $p$-closed.
  If $v'$ is not the trivial valuation,
  then Lemma \ref{lem:imperfect-implies-uniformiser} and Lemma \ref{lem:finite-residue}
  imply that $Kv'$ is finite, which is absurd.
  Therefore $v'$ must be trivial, and $vK$ is isomorphic to $\Z$.
\end{proof}

\begin{lemma}\label{lem:p-th-power-classes-isom}
  In the situation of the previous lemma,
  the natural group homomorphism $K^\times/p \to \widehat{K}^\times/p$ is an isomorphism,
  where $\widehat{K}$ is the completion of $K$ with respect to $v$.
\end{lemma}
\begin{proof}
  Let $\pi \in K$ be a uniformiser.
  The completion $\widehat{K}$ is isomorphic to $\F_q\pow{t}$ for suitable $\F_q = Kv$,
  where the isomorphism sends $\pi$ to $t$.
  We obtain the following diagram:
  \[ \xymatrix{
      K^\times/p  \times K/\wp(K) \ar@<-5ex>[d] \ar@<4ex>[d] \ar[r] & \Br(K)[p] \ar[d] \\
      \widehat{K}^\times/p \times \widehat{K}/\wp(\widehat{K}) \ar[r] & \Br(\widehat{K})[p]
    }\]
  By the construction of the pairing, this diagram is commutative.
  The map $K/\wp(K) \to \widehat{K}/\wp(\widehat{K})$ is surjective since $K$ is dense in $\widehat{K}$ and $\wp(\widehat{K})$ is open,
  as $\wp(\widehat{K}) \supseteq \widehat{\mathfrak{m}}_v$ by Hensel's Lemma.
  It is also injective since $\wp(\widehat{K}) = \wp(K) + \widehat{\mathfrak{m}}_v$,
  and $\widehat{\mathfrak{m}}_v \cap K = \mathfrak{m}_v \subseteq \wp(K)$ by the $p$-henselianity assumption.
  The map $\Br(K)[p] \to \Br(\widehat{K})[p]$ is an isomorphism
  since both sides are isomorphic to $\Z/p$
  and the image of $\pair{\mathcal{O}_v}{\pi}$ in $\Br(\widehat{K})[p]$ is non-zero
  by standard facts on the Brauer group of local fields, cf.\ \cite[Remark 6.3.6]{GilleSzamuely_2nd}.

  Using the pairing axiom (Pair) for both $K$ and $\widehat{K}$
  (where it holds by Proposition \ref{prop:local-satisfies-axioms}),
  we obtain that $K^\times/p \to \widehat{K}^\times/p$ is an isomorphism.
\end{proof}

\begin{proposition}\label{prop:field-is-local}
  Suppose, beyond the standing assumptions of this section, that $[K : K^p] \leq p$.
  Let $v$ be a non-trivial $p$-henselian valuation on $K$ with $Kv$ not $p$-closed.
  Then $(K, v)$ is a complete discretely valued field with finite residue field, i.e.\ a local field.
\end{proposition}
\begin{proof}
  Applying Lemma \ref{lem:discrete},
  it only remains to show that $K = \widehat{K}$, i.e.\ $(K,v)$ is complete.
  Let $\pi$ be a uniformiser of $(K,v)$.
  For $x \in \widehat{K}$, consider the element $1 + \pi x^p$.
  Since $K^\times/p \to \widehat{K}^\times/p$ is an isomorphism by Lemma \ref{lem:p-th-power-classes-isom},
  there exist $y \in K$ and $z \in \widehat{K}^\times$ with $y = z^p(1 + \pi x^p)$.
  By the assumption $[K : K^p] \leq p$, the element $\pi$ of $K$ is a $p$-basis, and so we may write $y = y_0^p + \pi y_1^p + \dotsb + \pi^{p-1} y_{p-1}^p$ with elements $y_i \in K$.
  We have thus obtained the equation
  \[ y_0^p + \pi y_1^p + \dotsb + \pi^{p-1} y_{p-1}^p = z^p + \pi z^p x^p .\]
  Since $\pi$ is also a $p$-basis of the field $\widehat{K}$, we can compare coefficients to obtain $z = y_0 \in K$ and $y_1 = zx$, so in particular $x = y_1/y_0 \in K$, as desired.  
\end{proof}

\begin{remark}\label{rem:condition-imperfection-needed}
  The additional assumption $[K : K^p] \leq p$ cannot be dispensed with, as the following construction due to Arno Fehm shows.

  Consider the local field $\F_q\pow{t}$ and enumerate its elements as $(x_\alpha)_{\alpha < 2^{\aleph_0}}$.
  Let $s \in \F_q\pow{t}$ be an element such that $s$ and $t$ are algebraically independent.
  We construct an increasing chain of subfields $K_\beta$ of $\F_q\pow{t}$ indexed by $\beta < 2^{\aleph_0}$,
  where $K_0 = \F_q(t)$, every $K_\beta$ has cardinality strictly less than $2^{\aleph_0}$,
  $K_\beta$ is relatively algebraically closed in $\F_q\pow{t}$,
  the image of the map $K_\beta^\times/p \to \F_q\pow{t}^\times/p$ contains the $p$-th power classes of all $x_\alpha$ with $\alpha < \beta$,
  and $s$ is transcendental over $K_\beta$.
  This can be done straightforwardly by transfinite induction:
  At limit ordinals $\beta$ one takes $K_\beta := \bigcup_{\gamma < \beta} K_\gamma$,
  and for a successor ordinal $\beta+1$ one takes $K_{\beta+1}$ to be the relative algebraic closure of $K_\beta(x)$ for a suitably chosen element $x$ in the $p$-th power class $x_\beta \F_q\pow{t}^{\times p}$ of $x_\beta$.
  The only constraint is ensuring that $s$ remains transcendental over $K_{\beta+1}$;
  to do so, it suffices to pick $x$ not algebraic over $K_\beta(s)$,
  which can always be done since $x_\beta \F_q\pow{t}^{\times p}$ has strictly larger cardinality than $K_\beta(s)$.

  Now the field $K = \bigcup_{\beta < 2^{\aleph_0}} K_\beta$ is relatively algebraically closed in $\F_q\pow{t}$, the map $K^\times/p \to \F_q\pow{t}^\times/p$ is an isomorphism (surjectivity is by construction, and injectivity follows from relative algebraic closedness), but $K$ does not contain $s$.
  The restriction map of absolute Galois groups $\G_{\F_q\pow{t}} \to \G_K$ is an isomorphism:
  here surjectivity follows from relative algebraic closedness,
  and injectivity holds since the restriction map $\G_{\F_q\pow{t}} \to \G_{\F_q(t)}$ is injective by Krasner's Lemma \cite[Corollary 18.5.3]{Efrat_ValuationsOrderingsMilnorKTheory}.
  Furthermore, the restriction map of Brauer groups $\Br(K) \to \Br(\F_q\pow{t})$ is an isomorphism,
  as $\Br(K)$ is known from the general theory of the Brauer group
  \cite[Corollaire 2.3]{Grothendieck_Brauer3}.

  Hence $K$ is a counterexample to Proposition \ref{prop:field-is-local} when the condition $[K : K^p] \leq p$ is dropped, and even to Theorem \ref{thm:intro-basic} without the axiom (Imp).
  Indeed, $K$ is not isomorphic to a local field, since it carries a unique non-trivial henselian valuation (namely the restriction of the $t$-adic valuation on $\F_q\pow{t}$) but is not complete.
\end{remark}

\begin{remark}\label{rem:alternative-to-imperfection-axiom}
  The condition $[K : K^p] \leq p$ can be replaced by some alternatives.
  For instance, let us instead suppose that every purely inseparable extension $L/K$ of degree $p$ also satisfies the axioms (Brau) and (Pair).
  Then by applying Lemma \ref{lem:p-th-power-classes-isom}, we see that for every such $L$ the map $L^\times/p \to (L\widehat{K})^\times/p$ must be an isomorphism, and so in particular $L$ is relatively inseparably closed in $L\widehat{K}$.
  For a uniformiser $\pi$ of $K$ we in particular obtain that $K(\sqrt[p]{\pi})$ is relatively inseparably closed in $\widehat{K}(\sqrt[p]{\pi}) = \widehat{K}^{1/p}$, implying that $K(\sqrt[p]{\pi}) = K^{1/p}$, so $[K : K^p] = [K^{1/p} : K] = p$.
\end{remark}

We can now prove the main theorem from the introduction.
\begin{proof}[Proof of Theorem \ref{thm:intro-basic}]
  Let $K$ be a field satisfying all the hypotheses.
  Then by Corollary \ref{cor:val-exists} and Proposition \ref{prop:field-is-local}, $K$ is a local field of characteristic $p$, i.e.\ $K \cong \F_{q'}\pow{t}$ for some power $q'$ of $p$.
  Since the $t$-adic valuation on $\F_{q'}\pow{t}$ is the unique non-trivial henselian valuation on this field by F.~K.~Schmidt's Theorem \cite[Theorem 4.4.1]{EnglerPrestel}, Proposition \ref{prop:EF-replacement} shows that we must have $q' = q$.
\end{proof}

\section{Sharpenings using quotients of the absolute Galois group}
\label{sec:sharpenings-quotients}

Let us observe the ways in which the axiom $\G_K \cong \G_{\F_q\pow{t}}$ was used in the proof of Theorem \ref{thm:intro-basic}.
We established that $K$ has characteristic $p$ in Lemma \ref{lem:char-p},
using that the $p$-Sylow subgroups of $\G_K$ are projective,
and used Corollary \ref{cor:val-exists} to obtain a valuation on $K$.
In Section \ref{sec:pinning-down} we used throughout that $K$ has countably many $\Z/p$-extensions,
i.e.\ that $\Hom(\G_K, \Z/p)$ is countable.
Lastly, we finished the proof by observing non-isomorphic local fields of characteristic $p$ have non-isomorphic absolute Galois groups.

The last two of these four uses actually need much weaker information than the isomorphism of full absolute Galois groups $\G_K \cong \G_{\F_q\pow{t}}$.
Indeed, the information about $\Z/p$-extensions is contained in $\G_K^{\mathrm{ab}}$,
and likewise non-isomorphic local fields of characteristic $p$ have non-isomorphic abelianised absolute Galois groups by \cite[Proposition 7.5.9]{NSW}.

In this section we want to show that also Lemma \ref{lem:char-p} and Corollary \ref{cor:val-exists} have versions in which the hypothesis only refers to certain quotients of $\G_K$.
The main tool to achieve this is \cite{CheboluEfratMinac_QuotientsDetermineCohomology}.
The proofs here make slightly more serious use of Galois cohomology and Milnor $K$-groups than in the preceding sections,
so we assume a passing familiarity with the definitions and basic properties on the part of the reader.

Following \cite[§4]{CheboluEfratMinac_QuotientsDetermineCohomology}, given a profinite group $G$ and a prime $p$, we write $G^{[3,p]}$ for the quotient $G/((G^p[G,G])^p [G^p[G,G], G])$,
i.e.\ the quotient of $G$ by the third term in its lower $p$-central series (also called descending $p$-central sequence).
Observe that $G^{[3,p]}$ is a pro-$p$ group which is nilpotent of class $2$ as a central extension of $G/(G^p[G,G])$ by $G^p[G,G]/((G^p[G,G])^p [G^p[G,G], G])$.

\begin{lemma}\label{lem:char-p-sharpened}
  Let $p$ be a prime number,
  and let $K$ be a field such that $\Br(K)[p] \neq 0$ and
  the absolute Galois group $\G_K$ satisfies the following:
  For every open normal subgroup $H$ of $\G_K$ with cyclic quotient
  $\G_K/H$ of order dividing $p-1$,
  $H^{[3,p]} \cong F^{[3,p]}$
  for some free pro-$p$ group $F$.
  Then $K$ has characteristic $p$.

  Conversely, every field $K$ of characteristic $p$ satisfies
  the condition on $\G_K$.
\end{lemma}
\begin{proof}
  Suppose first that $K$ is of characteristic $p$.
  Let $H \leq \G_K$ be a closed subgroup and $L/K$ the corresponding algebraic extension.
  Then the maximal pro-$p$ quotient $F := H(p) = \G_L(p)$ is free pro-$p$ \cite[Corollary 6.1.3]{NSW}.
  Since the quotient $H^{[3,p]}$ of $H$ is pro-$p$,
  the epimorphism $H \to F$ induces an isomorphism $H^{[3,p]} \cong F^{[3,p]}$.
  This proves that $\G_K$ satisfies the given condition.

  Suppose conversely that $K$ is a field of characteristic not $p$
  such that the condition on $\G_K$ is satisfied;
  we shall show that $\Br(K)[p] = 0$.
  Let $L = K(\zeta_p)$, and $H = \G_L \leq \G_K$.
  The group $G_K/H \cong \Gal(L/K)$ embeds into $(\Z/p)^\times$
  and hence is cyclic of order dividing $p-1$.
  Let $F$ be a free pro-$p$ group satisfying $H^{[3,p]} \cong F^{[3,p]}$.
  By \cite[Corollary 23.1.2]{FriedJarden} there exists a field $L'$ containing $\Q(\zeta_p)$ with absolute Galois group $\G_{L'}$ isomorphic to $F$.
  Now \cite[Theorem A]{CheboluEfratMinac_QuotientsDetermineCohomology} implies that
  $H^2(\G_L, \Z/p) \cong H^2(\G_{L'}, \Z/p)$,
  and the latter group vanishes since $\G_{L'} \cong F$ is projective.
  Therefore $H^2(\G_L, \mu_p) \cong H^2(\G_L, \Z/p)$ also vanishes.
  Since the restriction map $H^2(\G_K, \mu_p) \to H^2(\G_L, \mu_p)$ is injective
  as $L/K$ is of degree coprime to $p$,
  $\Br(K)[p] = H^2(\G_K, \mu_p)$ must be trivial.
\end{proof}

Let $\mathfrak{C}$ be the class of finite groups which are extensions of an abelian group by a nilpotent group of nilpotence class at most $2$.
In other words, $\mathfrak{C}$ consists of finite groups $G$ for which there exists a normal subgroup $H$ such that $G/H$ is abelian, and the quotient of $H$ by its centre is abelian.
An arbitrary profinite group $G$ always has a maximal pro-$\mathfrak{C}$ quotient $G(\mathfrak{C})$ in the usual sense \cite[Definition 17.3.2]{FriedJarden}.
\begin{corollary}\label{cor:char-p-sharpened}
  Let $K$ and $K'$ be two fields with $\G_K(\mathfrak{C}) \cong \G_{K'}(\mathfrak{C})$.
  If $K'$ has characteristic $p > 0$ and $\Br(K)[p] \neq 0$,
  then $K$ also has characteristic $p$.
\end{corollary}
\begin{proof}
  To any open normal subgroup $H$ of $\G_K$ with cyclic quotient of order dividing $p-1$,
  a choice of isomorphism $\G_K(\mathfrak{C}) \cong \G_{K'}(\mathfrak{C})$ associates
  an open normal subgroup $H'$ of $\G_{K'}$ with cyclic quotient of the same order
  such that the maximal two-step nilpotent quotients of $H$ and $H'$ are isomorphic.
  In particular $H^{[3,p]} \cong H'^{[3,p]}$.
  Thus the lemma applies.
\end{proof}

We now give a sharpened version of Corollary \ref{cor:val-exists},
concerning the existence of a valuation with residue field which is not $p$-closed.
Compared to there, the Galois-theoretic assumption is weakened.
On the other hand, we also do not obtain full henselianity,
but only $p$-henselianity, which was a crucial property in Section~\ref{sec:pinning-down}.
\begin{proposition}\label{prop:val-exists-sharpened}
  Let $p$ be a prime, $q > 1$ a power of $p$, and
  $K$ a field of characteristic $p$
  such that $\G_K(\mathfrak{C}) \cong \G_{\F_q\pow{t}}(\mathfrak{C})$.
  Then $K$ carries a non-trivial $p$-henselian valuation $v$ whose residue field $Kv$ is not $p$-closed.
\end{proposition}
Our proof loosely follows \cite[Proposition 2.3]{EfratFesenko},
although somewhat more work is required here.
We will use the notion of $l$-henselianity for a valuation $v$ on a field $K$,
where $l$ is an arbitrary prime number \cite[§4.2]{EnglerPrestel}.
So far, we have only considered this for $l$ the characteristic of $K$.
In general, a valuation $v$ on $K$ is $l$-henselian
if it extends uniquely to every Galois extension of $l$-power degree,
or equivalently to every $\Z/l$-extension \cite[Theorem~4.2.2]{EnglerPrestel}.
This holds if and only if Hensel's Lemma holds for every polynomial whose roots
generate a Galois extension of $l$-power degree \cite[Theorem 4.2.3]{EnglerPrestel}.
If the residue field $Kv$ is of characteristic not equal to $l$ and
$K$ contains the $l$-th roots of unity,
this is equivalent to $1 + \mathfrak{m}_v$ being contained in the set of $l$-th powers \cite[Corollary 4.2.4]{EnglerPrestel}.
This is an analogue of the previously used fact that,
when $l$ is the characteristic of $K$, $v$ is $l$-henselian
if and only if $\mathfrak{m}_v$ is contained in the image of the polynomial $X^l-X$ \cite{ChatzidakisPerera_pHenselianity}.

The following can be seen as a criterion for $p$-henselianity.
It is inspired by the criterion for full henselianity \cite[Proposition~2.1]{Efrat_GaloisCharacterization}.
\begin{lemma}\label{lem:criterion-p-henselian}
  Let $v$ be a valuation on a field $K$.
  Let $l$ be a prime number not equal to the characteristic of $Kv$
  such that $Kv$ is not $l$-closed and $vK$ is not $l$-divisible.
  Let $p$ also be a prime number.
  If $v$ is not $p$-henselian, then some finite abelian extension of $K$ has a $(\Z/l)^{p+1}$-extension.
\end{lemma}
\begin{proof}
  There is a $\Z/p$-extension $K'/K$ to which $v$ does not prolong uniquely.
  By \cite[Theorem 3.3.3]{EnglerPrestel} there are precisely $p$ prolongations $w_1, \dotsc, w_p$ of $v|_K$ to $K'$,
  necessarily with residue field $K'w_i = Kv$.
  These prolongations are incomparable \cite[Lemma 3.2.8]{EnglerPrestel}.
  Prolong each $w_i$ as $w_i'$ to $K'(\zeta_l)$ in some way.
  Noting that $[K'(\zeta_l) : K'] < l$, \cite[Corollary 3.2.3]{EnglerPrestel} implies
  that each residue field $K'(\zeta_l)w_i'$ is an extension of $K'w_i = Kv$ of degree less than $l$.
  Therefore none of the fields $K'(\zeta_l)w_i'$ can be $l$-closed,
  since a $\Z/l$-extension of $K'w_i = Kv$ cannot trivialise over $K'(\zeta_l)w_i'$.

  Let $R \subseteq K'(\zeta_l)$ be the intersection of the valuation rings $\mathcal{O}_{w_i'}$.
  By \cite[Theorem 3.2.7]{EnglerPrestel},
  the natural ring homomorphism $R \to \prod_{i=1}^p K'(\zeta_l)w_i'$ is surjective,
  so in particular $R^\times/l$ surjects onto the product of the $(K'(\zeta_l)w_i')^\times/l$.
  Since the fields $K'(\zeta_l)w_i'$ are not $l$-closed,
  Kummer theory implies that the groups $(K'(\zeta_l)w_i')^\times/l$ are not trivial.
  Therefore $R^\times/l$ has size at least $l^p$.
  On the other hand $K'(\zeta_l)^\times/R^\times$ surjects onto $K'(\zeta_l)^\times/\mathcal{O}_{w_1'}^\times$,
  which is precisely the value group of $w_1'$.
  This value group contains the value group of $v$ as a subgroup of finite index,
  and therefore is not $l$-divisible since the value group of $v$ is not $l$-divisible
  (see for instance \cite[Lemma~1.1.4~(d)]{Efrat_ValuationsOrderingsMilnorKTheory}).
  The short exact sequence $1 \to R^\times \to K'(\zeta_l)^\times \to K'(\zeta_l)^\times/R^\times \to 1$ induces an exact sequence
  $1 \to R^\times/l \to K'(\zeta_l)^\times/l \to (K'(\zeta_l)^\times/R^\times)/l \to 1$ since $K'(\zeta_l)^\times/R^\times$ is torsion-free.
  We deduce that $K'(\zeta_l)^\times/l$ has size at least $l^{p+1}$,
  and so the field $K'(\zeta_l)$ has a $(\Z/l)^{p+1}$-extension by Kummer theory.
  The extension $K'(\zeta_l)/K$ is abelian as the composite of the two abelian extensions $K'/K$ and $K(\zeta_l)/K$.
\end{proof}

We also require the following simple lemma for the proof of Proposition~\ref{prop:val-exists-sharpened}.
\begin{lemma}\label{lem:exists-l}
  Let $p$ be a prime and $q > 1$ a power of $p$.
  There exist a prime $l \neq p$ and $m \geq 1$ such that
  $l \nmid [\F_q : \F_p]$,
  $\zeta_l \not\in \F_q$,
  $[\F_p(\zeta_l) : \F_p]$ is a power of $p$,
  and $[\F_q(\zeta_{l^{m+1}}) : \F_q(\zeta_l)] = [\F_p(\zeta_{l^{m+1}}) : \F_p(\zeta_l)]$ is a power of $l$ not equal to $1$.
\end{lemma}
\begin{proof}
  For any natural number $n$, the quotient
  \[ \frac{p^{p^{n+1}}-1}{p^{p^n}-1} = \sum_{i=0}^{p-1} (p^{p^n})^i \]
  is congruent to $p$ modulo $p^{p^n}-1$,
  and therefore certainly coprime to $p^{p^n}-1$.
  It follows that there are infinitely many primes $l$ dividing $p^{p^n}-1$ for some $n$.
  Pick such a prime $l$ for which $l \nmid q-1$, $l \nmid [\F_q : \F_p]$.
  We have $\zeta_l \not\in \F_q$ since the order $\lvert \F_q^\times\rvert = q-1$ is not a multiple of $l$.
  Let $m \geq 1$ be maximal such that $l^m \mid p^{p^n}-1$,
  so that the cyclic group $\F_{p^{p^n}}^\times$ has elements of order $l^m$, but none of order $l^{m+1}$.
  This means that $\zeta_l, \zeta_{l^m} \in \F_{p^{p^n}}$, but $\zeta_{l^{m+1}} \not\in \F_{p^{p^n}}$.
  The degree $[\F_p(\zeta_l) : \F_p] | [\F_{p^{p^n}} : \F_p] = p^n$ is a power of $p$.

  The degree $[\F_p(\zeta_{l^{m+1}}) : \F_p(\zeta_l)]$ is a power of $l$
  since the extension arises as an iterated Kummer extension by successively adjoining $l$-th roots of the element $\zeta_l$.
  The degree is also not $1$ since $\F_p(\zeta_l) \subseteq \F_{p^{p^n}} \not\ni \zeta_{l^{m+1}}$.
  Lastly, $[\F_p(\zeta_{l^{m+1}}) : \F_p(\zeta_l)] = [\F_q(\zeta_{l^{m+1}}) : \F_q(\zeta_l)]$ since
  $\F_q(\zeta_l)$ and $\F_p(\zeta_{l^{m+1}})$ are Galois extensions of $\F_p(\zeta_l)$ of coprime degrees.
\end{proof}

\begin{proof}[Proof of Proposition \ref{prop:val-exists-sharpened}]
  Choose $l$ and $m$ as in Lemma \ref{lem:exists-l}.
  Fix an isomorphism $\sigma \colon \G_K(\mathfrak{C}) \to \G_{\F_q\pow{t}}(\mathfrak{C})$.
  Let $L_0/K$ be the cyclic extension corresponding to $\F_q\pow{t}(\zeta_l)/\F_q\pow{t}$ via $\sigma$,
  and let $L = L_0(\zeta_l)$.
  Both $L_0/K$ and $K(\zeta_l)/K$ are abelian extensions of $p$-power degree by the choice of $l$,
  and hence so is their compositum $L/K$.
  Let $F/\F_q\pow{t}$ be the extension corresponding to $L$.
  Since $L/K$ and $F/\F_q\pow{t}$ are abelian extensions,
  the Galois groups $\G_L$ and $\G_F$ have isomorphic maximal $2$-step nilpotent quotients by the construction of $\mathfrak{C}$.
  In particular, the graded rings $H^\ast(\G_L, \Z/l) \cong \KM_\ast(L)/l$ and $H^\ast(\G_F, \Z/l) \cong \KM_\ast(F)/l$ are isomorphic by \cite[Theorem A]{CheboluEfratMinac_QuotientsDetermineCohomology}.
  Thus $L^\times/l = \KM_1(L)/l \cong F^\times/l \cong (\Z/l)^2$ and $\KM_2(L)/l \cong H^2(\G_F, \Z/l) = \Br(F)[l] \cong \Z/l$.
  Since the ring $\KM_\ast(L)/l$ is skew-commutative,
  the multiplication map $L^\times/l \times L^\times/l \to \KM_2(L)/l$ is (surjective and) skew-commutative,
  so $\KM_2(L)/l$ is naturally identified with the second exterior power $\bigwedge^2_{\F_l} L^\times/l$.
  Thus the product of any $\F_l$-linearly independent elements of $L^\times/l$ in $\KM_2(L)/l$ is non-zero.
  Applying \cite[Main Theorem]{Efrat_ConstructionValuations}\footnote{Cf.\
    also \cite[Theorem 2.1]{EfratFesenko},
    but beware the misprint $u(l) \neq 0$ in place of the correct $u(l) = 0$.}
  with the prime $l$ and $T = L^{\times l}$,
  it follows that $L$ carries a valuation $v$ such that
  $(vL : l \cdot vL) \geq l$, the residue field $Lv$ has characteristic not $l$,
  and $1 + \mathfrak{m}_v \subseteq L^{\times l}$, so $(L,v)$ is $l$-henselian.

  We claim that the residue field $Lv$ does not contain $\zeta_{l^{m+1}}$.
  If it does, then by $l$-henselianity $\zeta_{l^{m+1}} \in L$,
  since $\zeta_{l^{m+1}}$ has $l$-power degree over $L \supseteq \F_p(\zeta_l)$.
  We know that $L^\times/l \cong (\Z/l)^2$, say generated by the classes of $a, b \in L^\times$.
  Then by Kummer theory both $L(a^{1/l^{m+1}})$ and $L(b^{1/l^{m+1}})$ are $\Z/l^{m+1}$-extensions of $L$.
  Their intersection is $L$, since their minimal non-trivial subextensions $L(a^{1/l})$ and $L(b^{1/l})$
  are distinct.
  Therefore $L$ has a $(\Z/l^{m+1})^2$-extension, namely $L(a^{1/l^{m+1}}, b^{1/l^{m+1}})$,
  and accordingly $F$ also has a $(\Z/l^{m+1})^2$-extension.
  By \cite[Proposition~7.5.9]{NSW}, this implies that the residue field of $F$ contains $\zeta_{l^{m+1}}$.
  Since $F/\F_q\pow{t}$ is a Galois extension of $p$-power degree,
  its residue field is a $p$-power degree extension of $\F_q$.
  However, $\zeta_{l^{m+1}}$ has degree divisible by $l$ over $\F_q$ by the choice of $l$.
  This contradiction shows that $\zeta_{l^{m+1}} \not\in Lv$.
  The degree of $\zeta_{l^{m+1}}$ over $Lv$ is some power of $l$,
  since $\zeta_{l^{m+1}}$ has $l$-power degree over $\F_p(\zeta_l) \subseteq Lv$.

  We now investigate the restriction of $v$ to $K$.
  Observe that $\zeta_{l^{m+1}}$ generates an abelian extension of $Kv$ whose degree is divisible by $l$
  since $\zeta_{l^{m+1}}$ has degree divisible by $l$ even over $Lv$,
  and therefore $Kv$ is not $l$-closed.
  If $v|_K$ is not $p$-henselian, then by Lemma~\ref{lem:criterion-p-henselian}
  some finite abelian extension $K'/K$ has a $(\Z/l)^{p+1}$-extension.
  However, the extension of $\F_q\pow{t}$ corresponding to it under $\sigma$
  (as indeed any finite extension of $\F_q\pow{t}$)
  does not have a $(\Z/l)^{p+1}$-extension \cite[Proposition 7.5.9]{NSW}.
  This contradiction shows the $p$-henselianity of $v|_K$.

  Lastly, let us show that $\zeta_l \not\in Kv$.
  Supposing for a contradiction that $\zeta_l \in Kv$, then also $\zeta_l \in K$
  as $\zeta_l$ has $p$-power degree over $\F_p$ and $v|_K$ is $p$-henselian.
  Therefore $K$ has a $(\Z/l)^2$-extension,
  given by on the one hand adjoining an $l$-th root of an element $x \in K^\times$
  with $v(x) \in vK$ not divisible $l$,
  and on the other hand adjoining an $l$-th power root of unity
  (since we know that $\zeta_{l^{m+1}}$ does not lie in $Lv$ and therefore not in $K$).
  However, $\F_q\pow{t}$ does not have a $(\Z/l)^2$-extension,
  so given the isomorphism $\sigma$ we once more obtain a contradiction.
  Therefore $Kv(\zeta_l)/Kv$ is a non-trivial Galois extension
  which is of $p$-power degree by the choice of $l$.
  Thus $Kv$ is not $p$-closed.
\end{proof}

\begin{remark}\label{rem:strong-val-exists-sharpened}
  Inspection of the proof shows that the valuation $v$ is not only $p$-henselian,
  but in fact $r$-henselian for every prime number $r$,
  since Lemma~\ref{lem:criterion-p-henselian} can be applied for primes other than the characteristic $p$.
\end{remark}

We can now deduce the following strengthened version of Theorem \ref{thm:intro-basic}.
\begin{theorem}\label{thm:main-sharpened}
  Let $p$ be a prime number, $q = p^n$ for some $n \geq 1$.
  Let $K$ be a field satisfying the following axioms:
  \begin{enumerate}
  \item[(Gal\textsubscript{0})]
    The maximal pro-$\mathfrak{C}$-quotients $\G_K(\mathfrak{C})$ and $\G_{\F_q\pow{t}}(\mathfrak{C})$ are isomorphic,
    where $\mathfrak{C}$ is (as above) the class of finite groups which are extensions of an abelian group by a nilpotent group of nilpotence class at most $2$;
  \item[(Imp)] $K$ has exponent of imperfection at most $1$;
  \item[(Brau)] $\Br(K)[p] \cong \Z/p\Z$;
  \item[(Pair)] the natural pairing $\Hom(\G_K, \Z/p) \times K^\times/p \to \Br(K)[p]$ induces an isomorphism $K^\times/p \cong \Hom(\Hom(\G_K, \Z/p), \Br(K)[p])$.
  \end{enumerate}
  Then $K$ is isomorphic to the local field $\F_q\pow{t}$.
\end{theorem}
\begin{proof}
  By Corollary \ref{cor:char-p-sharpened}, $K$ has characteristic $p$.
  By Proposition \ref{prop:val-exists-sharpened}, $K$ carries a non-trivial $p$-henselian valuation $v$ whose residue field is not $p$-closed.
  Now Proposition \ref{prop:field-is-local} shows that $K$ is isomorphic to $\F_{q'}\pow{t}$ for some power $q'$ of $p$.
  If we had $q' \neq q$, then $\G_{\F_{q'}\pow{t}}^{\mathrm{ab}} \not\cong \G_{\F_q\pow{t}}^{\mathrm{ab}}$ by \cite[Proposition 7.5.9]{NSW}, and so in particular $\G_{\F_{q'}\pow{t}}(\mathfrak{C}) \not\cong \G_{\F_q\pow{t}}(\mathfrak{C})$.
  Therefore $q' = q$ and $K \cong \F_q\pow{t}$.
\end{proof}

The choice of the class of groups $\mathfrak{C}$ is not quite optimal.
It is already apparent from the proofs above that it could be shrunk by some amount,
strengthening the statement of Theorem \ref{thm:main-sharpened}.
However, this would yield a class less easy to describe than $\mathfrak{C}$.

\bibliographystyle{amsalpha}
\bibliography{Bib/Bibliography}

\end{document}